\documentclass[12pt,english]{article}
\usepackage{babel}
\usepackage[cp1250]{inputenc}
\usepackage[T1]{fontenc}
\usepackage{amsfonts}
\usepackage{amsmath}
\usepackage{amsthm}
\usepackage{mathrsfs}
\usepackage{hyperref}
\usepackage{enumerate}
\usepackage{graphicx}
\usepackage{pictex}

\usepackage{xcolor}

\newcommand{\eqq}[2]{\begin{equation}  #1  \label{#2}
\end{equation}    }

\def\divv{\operatorname{div}}
\newcommand{\hd}{\hspace{0.2cm}}
\newcommand{\no}{\noindent}
\newcommand{\m}[1]{\mbox{#1}}
\newcommand{\rr}{\mathbb{R}}
\newcommand{\NN}{\mathbb{N}}

\newcommand{\lap}{\Delta}
\newcommand{\lapk}{\Delta^{2}}

\newcommand{\nal}{\nabla \Delta}
\newcommand{\la}{\lambda}
\newcommand{\laf}{\tilde{\lambda}}
\newcommand{\nab}{\nabla}
\newcommand{\nabd}{\nabla^{2}}
\newcommand{\nabt}{\nabla^{3}}

\newcommand{\jd}{\frac{1}{2}}

\newcommand{\ddt}{\frac{d}{dt}}
\newcommand{\be}[1]{\left|  #1 \right|}
\newcommand{\bk}[1]{\left|  #1 \right|^{2}}

\newcommand{\nd}[1]{\| #1  \|_{2}}

\newcommand{\ns}[1]{\| #1  \|_{6}}
\newcommand{\nsk}[1]{\| #1  \|_{6}^{2}}

\newcommand{\nt}[1]{\| #1  \|_{3}}
\newcommand{\ntk}[1]{\| #1  \|_{3}^{2}}
\newcommand{\nc}[1]{\| #1  \|_{4}}
\newcommand{\nck}[1]{\| #1  \|_{4}^{2}}
\newcommand{\ncc}[1]{\| #1  \|_{4}^{4}}

\newcommand{\ndk}[1]{\| #1  \|_{2}^{2}}
\newcommand{\nif}[1]{\| #1  \|_{\infty}}
\newcommand{\nifk}[1]{\| #1  \|_{\infty}^{2}}

\newcommand{\nsod}[1]{\| #1 \|_{2,2} }
\newcommand{\nsot}[1]{\| #1 \|_{3,2} }
\newcommand{\nsotk}[1]{\| #1 \|_{3,2}^{2} }

\newcommand{\nsodk}[1]{\| #1 \|_{2,2}^{2} }
\newcommand{\nsodt}[1]{\| #1 \|_{2,2}^{3} }

\newtheorem{rem}{{\textbf {Remark}}}
\newtheorem{lem}{{\textbf {Lemma}}}

\newtheorem{theorem}{\textbf {Theorem}}

\newcommand{\vt}{\partial_{t} v}
\newcommand{\ot}{\partial_{t} \omega}
\newcommand{\bt}{\partial_{t} b}
\newcommand{\om}{\omega}
\newcommand{\Om}{\Omega}

\newcommand{\bno}{\frac{b}{\om}}
\newcommand{\dvk}{|D(v)|^{2}}

\newcommand{\bmi}{b_{\min}}

\newcommand{\omi}{\om_{\min}}
\newcommand{\oma}{\om_{\max}}
\newcommand{\bmt}{b_{\min}^{t}}
\newcommand{\bmtm}{(\bmt)^{-1}}

\newcommand{\omt}{\om_{\min}^{t}}
\newcommand{\ommt}{\om_{\max}^{t}}
\newcommand{\omtm}{(\omt)^{-1}}
\newcommand{\omtd}{(\omt)^{-2}}
\newcommand{\mnit}{\mu^{t}_{\min}}

\newcommand{\kd}{\kappa_{2}}

\newcommand{\bmit}{\frac{\bmi}{\left( 1 + \kd \oma t \right)^{\frac{1}{\kd}}}}

\newcommand{\omitt}{\frac{\omi}{1 + \kd \omi t}}
\newcommand{\omat}{\frac{\oma}{1 + \kd \oma t}}

\newcommand{\V}{\mathcal{V}}
\newcommand{\Vd}{\V^{2}}
\newcommand{\Vt}{\V^{3}}
\newcommand{\Vj}{\V^{1}}

\newcommand{\Vkd}{\dot{\V}_{\divv}}
\newcommand{\Vkdj}{\dot{\V}_{\divv}^{1}}
\newcommand{\Vkdd}{\Vkd^{2}}
\newcommand{\Vkdt}{\Vkd^{3}}

\newcommand{\n}[1]{ \left( #1  \right)}

\newcommand{\vl}{v^{l}}
\newcommand{\ol}{\om^{l}}
\newcommand{\bl}{b^{l}}

\newcommand{\vlt}{\partial_{t}v^{l}}
\newcommand{\olt}{\partial_{t}\om^{l}}
\newcommand{\blt}{\partial_{t}b^{l}}
\newcommand{\mul}{\mu^{l}}
\newcommand{\dvlk}{|D(\vl)|^{2}}

\newcommand{\cli}{c_{i}^{l}}
\newcommand{\eli}{e_{i}^{l}}
\newcommand{\dli}{d_{i}^{l}}

\newcommand{\clit}{\frac{d}{dt}c_{i}^{l}}

\newcommand{\dlit}{\frac{d}{dt}d_{i}^{l}}

\newcommand{\Pst}{\Psi_{t}}
\newcommand{\Pht}{\Phi_{t}}
\newcommand{\pst}{\psi_{t}}
\newcommand{\pht}{\phi_{t}}

\newcommand{\into}{\int_{\Omega}}
\newcommand{\dvl}{D(\vl)}
\newcommand{\nabk}{\nabla^{2}}

\newcommand{\phol}{\pht( \ol)}
\newcommand{\pholp}{\pht'( \ol)}

\newcommand{\psbl}{\pst( \bl)}
\newcommand{\psblp}{\pst'( \bl)}

\newcommand{\Phol}{\Pht( \ol)}

\newcommand{\Psbl}{\Pst( \bl)}

\newcommand{\ts}{t^{*}}
\newcommand{\cs}{C_{*}}
\newcommand{\cgs}{C^{*}}
\newcommand{\tskd}{\ts_{K,\delta}}

\newcommand{\sla}{\rightharpoonup}
\newcommand{\zb}{\longrightarrow}
\newcommand{\slg}{\overset{*}{\rightharpoonup}}

\newcommand{\mup}{\mu_{\Pst \Pht}}

\newcommand{\bp}{b_{+}}
\newcommand{\bm}{b_{-}}
\newcommand{\op}{\om_{+}}
\newcommand{\ou}{\om_{-}}

\newcommand{\itt}{\int_{0}^{t}}

\newcommand{\eqns}[1]{
\begin{eqnarray}
\begin{split}
#1
\end{split}
\end{eqnarray}
}
\newcommand{\eqnsl}[2]{
\begin{equation}
\label{#2}
\begin{split}
#1
\end{split}
\end{equation}
}

\newcommand{\ep}{\varepsilon}

\begin{document}

\title{Local in time solution to Kolmogorov's two-equation model of turbulence}

\author{Przemys\l aw Kosewski${}^{\ast}$, Adam Kubica\footnote{Department of Mathematics and Information Sciences, Warsaw University of Technology, ul. Koszykowa 75,
00-662 Warsaw, Poland, E-mail addresses: A.Kubica@mini.pw.edu.pl, P.Kosewski@mini.pw.edu.pl} }

\maketitle

\abstract{We prove the existence of local in time solution to  Kolmogorov's two-equation model of turbulence in three dimensional domain with periodic boundary conditions. We apply Galerkin method for appropriate truncated problem. Next, we obtain estimates for a limit of approximate solutions to ensure that it satisfies the original problem.   }

\vspace{0.3cm}

\no Keywords: Kolmogorov's two-equation model of turbulence, local in time solution, Galerkin method.

\vspace{0.2cm}


\section{Introduction}

Firstly,  we will provide a short introduction to turbulence modeling. We introduce an idea behind RANS (Reynolds Averaged Navier Stokes, see \cite{WILCOX}, \cite{RANS4}, \cite{RANS7}, \cite{TurbMod}) and explain the necessity of incorporating additional equations to model turbulence. Next, we will introduce Kolmogorov's two equation model and its connection to currently used turbulence models.

Turbulent flow is a fluid motion characterized by rapid changes in velocity and pressure. These fluctuations cause difficulties mainly in finding solutions using numerical methods, which require dense mesh and very short time steps to properly reproduce the turbulent flow. Additionally, turbulences appear to be self-similar and display a chaotic behaviour. This bolster a need for precise simulations.

The simplest idea that would decrease the apparent fluctuations of solutions is to consider the average value of the velocity and of the pressure. This is the case in RANS, where the average is taken with respect to the time. Now, let us decompose the velocity $v$ and \m{pressure $p$:}
\[
v(x,t) = \overline{v}(x,t) + \widetilde{v}(x,t) , \hd \hd
p(x,t) = \overline{p}(x,t) + \widetilde{p}(x,t) ,
\]
where $\overline{v}$, $\overline{p}$ are time-averaged values and $\widetilde{v}$, $\widetilde{p}$ are fluctuations. We substitute the decomposed functions into the Navier Stokes system and we get (for details see chapter 2 of \cite{WILCOX}).
\[
\partial_{t} \overline{v}
+ \overline{v} \cdot \nab \overline{v}
- \nu \divv D \overline{v}
+ \nab \overline{p}
=
- \divv \n{ \overline{\widetilde{v} \cdot  \widetilde{v}}}.
\]
The last term on the right hand side can be approximated by Boussinesq  approximation (see \cite{WILCOX})
\[
-\overline{\widetilde{v} \cdot \widetilde{v}}
=
\nu_T (\nab \overline{v} + \nab^T \overline{v} ) - \frac{2}{3} k I,
\]
where $ \nu_T = \frac{k}{\om}$, $k$ is the tubulent kinetic energy and  $\om$ is the  dissipation rate. Finaly, we obtain
\eqq{
\partial_{t} \overline{v}
+ \overline{v} \cdot \nab \overline{v}
- \nab \cdot \left( ( \nu + \nu_T)D \overline{v} \right)
+ \nab \left( \overline{p} + \frac{2}{3} k \right)
=
0.
}{ransEqv}
We see that to close the system we need to introduce additional equations for $\om$ and $k$.
For further details see \cite{WILCOX} and \cite{RANS7}.

Nowadays, $k-\varepsilon$ and $k - \omega$ are two of the most commonly used models to calculate $k$ and $\om$.
They bear a strong resemblance to Kolmogorov's turbulence model in the way they deal with diffusive terms. In both models, the equation on $k$  uses a  squared matrix norm of the symmetric gradient as a source term.

In 1941 A. N. Kolmogorov introduced following system of equations describing turbulent flow (\cite{Kolmog}, English translation in Appendix A \cite{Spal})
\eqq{\vt + \divv(v\otimes v) - 2 \nu_{0} \divv \n{ \bno  D(v)} = - \nabla p  , }{pa}
\eqq{ \ot + \divv(\om v ) - \kappa_{1} \divv \n{ \bno \nabla \om } = - \kappa_{2} \om^{2},   }{pb}
\eqq{\bt + \divv(b v   )   - \kappa_{3} \divv \n{  \bno \nabla b }  = - b \om + \kappa_{4} \bno \dvk,   }{pc}
\eqq{\divv{v} =0,}{pd}
where $v$ is the  mean velocity, $\om$ is the dissipation rate, $b$ represents  2/3 of the  mean kinetic energy, $p$ is the sum of the mean pressure and $b$. The novelty of Kolmogorov's formulation is that it no longer requires prior knowledge of the length scale (size of large eddies) - it can be calculated as $\frac{\sqrt{b}}{\om}$. Let us  notice that the proposed equation on velocity  highly resembles the equation (\ref{ransEqv}), which appeared in RANS. The $k-\varepsilon$ and $k - \omega$ systems provide similar equations for $ \om $ and $b$ with the addition of a source term in the equation for $ \om $.

The physical motivation of the proposed system can be found in \cite{Spal} and \cite{BuM}.  A mathematical analysis of the difficulties that occur in proving the existence of solutions of such a system can also be found in \cite{BuM}.

Now, we would like to discus the known mathematical results related to Kolmogorov's two-equation model of turbulence. There are two recent results  devoted to this problem: \cite{BuM} and  \cite{MiNa} (see the announcement \cite{MiNaa}) and our result is inspired by them. In the first one, the Authors consider the system in a bounded  $C^{1,1}$ domain with mixed boundary conditions for $b$ and $\om$ and a stick-slip boundary condition for the velocity $v$. In order to overcome the difficulties related with the last term on the right hand side of (\ref{pc}) the problem is reformulated and the quantity $E:=\frac{1}{2}|v|^{2}+ \frac{2\nu_{0}}{\kappa_{4}}b$ is introduced. Then, the equation (\ref{pc}) is replaced by
\[
\partial_{t}E+ \divv(v(E+p))- 2\nu_{0}\divv\left( \frac{\kappa_{3} b}{\kappa_{4}\om}\nab b+ \frac{b}{\om} D(v)v \right)+\frac{2\nu_{0}}{\kappa_{4}}b\om=0.
\]
The existence of global-in-time weak solution of the reformulated problem is established. It is also worth mentioning that in \cite{BuM} the assumption related to the initial value of $b$ tolerates the vanishing of $b_{0}$ in some points of the domain. More precisely, the existence of weak solution is proved under the conditions $b_{0}\in L^{1}$, $b_{0}>0$ a.e. and $\ln{b_{0}}\in L^{1}$.

In the article  \cite{MiNa} the Authors consider the system (\ref{pa})-(\ref{pd}) in a periodic domain.  The existence of global-in-time weak solution is proved, but due to the presence of the strongly nonlinear term $\bno \dvk$,  the weak form of equation (\ref{pc}) has to be corrected by a positive measure $\mu$, which is zero, if the  weak solution is sufficiently regular. There are also  estimates for $\om$ and $b$ (see (4.2) in \cite{MiNa}). These observations are crucial in our reasoning presented below. Concerning to the initial value of $b$, the assumption is  that $b_{0}$ is uniformly positive.

\section{Notation and main result.}

Assume that $\Omega = \prod_{i=1}^{3}(0,L_{i}) $, \hd $L_{i}$, $T>0$ and $\Omega^{T}=\Omega \times (0,T)$.  We  shall consider the following problem
\eqq{\vt + \divv(v\otimes v) -  \nu_{0} \divv \n{ \bno  D(v)} = - \nabla p  , }{a}
\eqq{ \ot + \divv(\om v ) - \kappa_{1} \divv \n{ \bno \nabla \om } = - \kappa_{2} \om^{2},   }{b}
\eqq{\bt + \divv(b v   )   - \kappa_{3} \divv \n{  \bno \nabla b }  = - b \om + \kappa_{4} \bno \dvk,   }{c}
\eqq{\divv{v} =0,}{d}
in $\Omega^{T}$ with periodic boundary condition on $\partial \Omega$ and initial condition
\eqq{v_{|t=0}= v_{0}, \hd \hd \om_{|t=0}= \om_{0}, \hd \hd b_{|t=0}= b_{0}.}{e}
Here $\nu_{0}, \kappa_{1}, \dots, \kappa_{4}$ are positive constants.  For simplicity, we assume further that all constants except $\kappa_{2}$ are equal to one. The reason is that the constant $\kappa_{2}$ plays an important role in the a priori estimates.

\no We shall show the local-in-time existence of regular solution of problem (\ref{a})-(\ref{e}) under some assumption imposed on the initial data. Namely, suppose that there exists positive numbers $\bmi$, $\omi$, $\oma$ such that
\eqq{0<\bmi\leq b_{0}(x) ,  }{f}
\eqq{0<\omi\leq \om_{0}(x) \leq \oma  }{g}
on $\Omega$ and we set
\eqq{
\begin{array}{c}
\bmt = \bmit,  \hd \hd
\omt = \omitt, \hd \\ \\
\ommt = \omat, \hd \hd
\mnit = \frac{1}{4}\frac{\bmt}{\ommt}.
\end{array} }{h}
If $m\in \mathbb{N}$, then   by $\V^{m}$ we denote the space of restrictions to $\Om$ of the functions, which belong to the space
\eqq{\{u \in H^{m}_{loc}(\rr^{3}): \hd  u(\cdot + kL_{i}e_{i}) = u(\cdot ) \hd \m{ for } \hd k\in \mathbb{Z}, \hd  i=1,2,3 \},}{j}
where $\{ e_{i}\}_{i=1}^{3}$ form a standard basis in $\rr^{3}$. Next, we define
\eqq{\Vkd^{m} = \{v\in \V^{m}:\hd \divv v =0, \hd \int_{\Om} vdx =0  \}.}{k}
We shall find the solution of the system (\ref{a})-(\ref{d}) such that $(v,\om,b)\in \mathcal{X}(T)$, where
\eqq{\mathcal{X}(T)= L^{2}(0,T;\Vkdt)\times L^{2}(0,T;\Vt))\times (L^{2}(0,T;\Vt) \cap (H^{1}(0,T;H^{1}(\Om)))^{5}.   }{i}
We shall denote by $\| \cdot \|_{k,2}$ the norm in the  Sobolev space, i.e.
\eqq{
\| f \|_{k,2} = (\ndk{\nab^{k} f } + \ndk{f})^{\frac{1}{2}},
}{norSob}
where $\| \cdot \|_{2}$ is $L^{2}$ norm on $\Om$.

Now, we introduce the notion of solution to the system (\ref{a})-(\ref{d}). We shall show that for any $v_{0}\in \Vkd^{2}$ and strictly positive $\om_{0}$, $b_{0}\in \V^{2}$ there exist  positive $T$ and $(v, \om, b)\in \mathcal{X}(T)$ such that
\eqq{(\vt, w) - (v\otimes v,\nabla w) +    \n{ \mu  D(v), D(w)} = 0  \hd \m{ for } \hd w\in \Vkd^{1}, }{aa}
\eqq{ (\ot, z) - (\om v, \nabla z  ) + \n{ \mu \nabla \om, \nabla z  } = - \kd (\om^{2}, z ) \hd \m{ for } \hd z\in \V^{1},    }{ab}
\eqq{(\bt, q) - (b v, \nab q   )   +\n{  \mu  \nabla b, \nabla q  }  = - (b \om , q) + ( \mu \dvk, q) \hd \m{ for } \hd q\in \V^{1},  }{ac}
for a.a. $t\in (0,T)$, where $\mu= \frac{b}{\om}$ and (\ref{e}) holds. Recall that  $D(v)$ denotes the symmetric part of $\nabla v$ and $(\cdot , \cdot )$ is the inner product in $L^{2}(\Om)$.

Our main result concerning the existence of local in time regular solutions is as follows.

\begin{theorem}
Suppose  that  $\om_{0}$, $b_{0}\in \V^{2}$,  $v_{0}\in \Vkd^{2}$ and  (\ref{f}), (\ref{g}) are satisfied. Then there exist positive $\ts$ and $(v,\om , b)\in \mathcal{X}(\ts)$ such that  (\ref{aa})-(\ref{ac}) hold for a.a. $t\in (0,\ts)$ and (\ref{e}) is satisfied.
Furthermore, for each $(x,t)\in \Om\times [0,\ts)$ the following estimates
\eqq{\frac{\omi}{1+\kd \omi t} \leq \om(x,t) \leq \frac{\oma}{1+\kd \oma t},}{newc}
\eqq{\frac{\bmi}{(1+\kd \oma t)^{\frac{1}{\kd}}}  \leq b(x,t) }{newd}
hold. The time of existence of the solution is estimated from below in the following sense: for each positive $\delta$ and compact $K\subseteq \{ (a,b,c): 0<a\leq b, \hd 0<c \}$ there exists positive $\tskd$, which depends only on $\kd, \Om, \delta$ and $K$ such that if
\eqq{  \nsodk{v_{0}}+\nsodk{\om_{0}}+\nsodk{b_{0}}\leq \delta  \hd \m{ and } \hd (\omi, \oma, \bmi)\in K,}{uniformly}
then $\ts\geq \tskd$. The Sobolev norm is defined by (\ref{norSob}).
\label{main}
\end{theorem}
We note that the last part of the theorem is needed for proving the existence of global in time solution for small data.  We address this issue in another paper.

In the next section we prove the above theorem by applying Galerkin method for an appropriate truncated problem. We obtain a priori estimates for the sequence of approximate solutions and by a weak-compactness argument we get a solution of the truncated problem. Finally, after proving some bounds for $\om$ and $b$ we deduce that the obtained solution satisfies the original system of equations.

\section{Proof of the main result}
The proof of theorem~\ref{main} is based on  Galerkin method. Hence, we need a basis of the spaces $\V^{1}$ and $\Vkd^{1}$. Let $\{ w_{i}\}_{i\in \NN}$ be a system of eigenfunctions  of Stokes operator in  $\Vkd^{1}$, which 
is  complete and orthogonal  in $\Vkd^{1}$ and orthonormal in $L^{2}(\Om)$ (see chap.~II.6 in \cite{Temam}). In particular, $\{ w_{i} \}_{i\in \NN}$ are smooth (see formula (6.17), chap. II in \cite{Temam}). By $\{ \lambda_{i} \}_{i\in \NN}$ we denote the corresponding system of eigenvalues.  Similarly, let $\{ z_{i} \}_{i\in \NN}$ be an complete and orthogonal system in $\V^{1}$, which is orthonormal in $L^{2}(\Om)$, which is obtained by taking eigenvectors  of the minus Laplace operator. The system of corresponding eigenvalues is denoted  by $\{ \laf_{i}  \}_{i\in \NN}$.
 We shall find approximate solutions of (\ref{aa})-(\ref{ac}) in the following form
\eqq{\vl(t, x)= \sum_{i=1}^{l}  \cli(t)w_{i}(x), \hd \hd  \ol(t, x)= \sum_{i=1}^{l}  \eli(t)z_{i}(x), \hd  \hd \bl(t, x)= \sum_{i=1}^{l}  \dli(t)z_{i}(x).}{ad}
We have to determine the coefficients $\{ \cli \}_{i=1}^{l}$, $\{ \eli \}_{i=1}^{l}$ and $\{ \dli \}_{i=1}^{l}$. In order to define an approximate problem we have to introduce a few auxiliary functions. For fixed $t>0$ we denote by $\Pst= \Pst(x)$ a smooth function such that
\eqq{
\Pst(x)= \left\{
\begin{array}{rll}
\frac{1}{2}\bmt & \m{ for } & x<\frac{1}{2}\bmt, \\ x &  \m{ for } & x\geq\bmt,\\
\end{array}  \right. }{defPsi}
where  $\bmt$ is defined by (\ref{h}).  We assume that the function $\Pst $ also satisfies
\eqq{ 0\leq \Pst'(x)\leq c_{0}, \hd |\Pst''(x)|\leq c_{0} (\bmt)^{-1},  }{estiPsi}
where, here and  $c_{0}$ is a constant independent on $x$ and $t$ (see in the appendix for details (formula (\ref{defPsik})). We  also need smooth functions $\Pht$, $\pst$ and $\pht$ such that
\eqq{ \Pht(x)=\left\{
\begin{array}{rll}
\frac{1}{2}\omt & \m{ for } & x< \frac{1}{2}\omt,\\
x  & \m{ for } & x \in [ \omt, \ommt], \\
2\ommt  & \m{ for } & x >  2\ommt,\\
\end{array} \right. }{defPhi}
\eqq{
\pst(x)= \left\{
\begin{array}{rll}
0 & \m{ for } & x<\frac{1}{2}\bmt, \\ x &  \m{ for } & x \geq \bmt,\\
\end{array}  \right. }{defpsi}
\eqq{ \pht(x)=\left\{
\begin{array}{rll}
0 & \m{ for } & x< \frac{1}{2}\omt,\\
x  & \m{ for } & x\geq \omt.\\
\end{array} \right. }{defphi}
We assume that these functions additionally satisfy
\eqq{ 0\leq \Pht'(x)\leq c_{0}, \hd |\Pht''(x)|\leq c_{0} (\omt)^{-1}, }{estiPhi}
\eqq{\pst(x)\leq x \hd \m{ for } \hd  x\geq 0  , \hd \hd  0\leq \pst'(x)\leq c_{0} \hd \m{ for } \hd  x\in \rr , }{estipsi}
\eqq{\pht(x)\leq x \hd \m{ for } \hd  x\geq 0 ,  \hd \hd  0\leq \pht'(x)\leq c_{0} \hd \m{ for } \hd  x\in \rr , }{estiphi}
for some constant $c_{0}$ (the construction of $\Pht$, $\pst$ and $\pht$ are similar to argument from the appendix).

An approximate solution will be found  in the form (\ref{ad}), where the coefficients $\{ \cli \}_{i=1}^{l}$, $\{ \eli \}_{i=1}^{l}$ and $\{ \dli \}_{i=1}^{l}$ are determined by the following truncated system
\eqq{(\vlt, w_{i}) - (\vl\otimes \vl,\nabla w_{i}) +    \n{ \mul  D(\vl), D(w_{i})} = 0 , }{ca}
\eqq{ (\olt, z_{i}) - (\ol \vl, \nabla z_{i}  ) + \n{ \mul \nabla \ol, \nabla z_{i}  } = - \kd (\pht^{2}(\ol), z_{i} ),  }{cb}
\eqq{(\blt, z_{i}) - (\bl \vl, \nab z_{i}   )   +\n{  \mul  \nabla \bl, \nabla z_{i}  }  = - (\pst(\bl)\pht( \ol) , z_{i}) + ( \mul \dvlk, z_{i}),  }{cc}
\[
c^{l}_{i}(0)= (v_{0},w_{i}), \hd e^{l}_{i}(0)= (\om_{0},z_{i}), \hd d^{l}_{i}(0)= (b_{0},z_{i}),
\]
where $i\in\{1,\dots , l \}$ and we denote
\eqq{\mul = \frac{\Pst(\bl)}{\Pht(\ol)}. }{cd}
In the computations below, the exponent $l$ systematically refers to this Galerkin approximation.
\begin{rem}
We emphasize that in order to control the second derivatives of approximated solutions we need the conditions (\ref{estiPhi})-(\ref{estiphi}). In particular, we can not apply piecewise linear functions.
\end{rem}

Firstly, we note that $\mul$ is positive and then, by standard ODE theory the system  (\ref{ca})-(\ref{cc}) has a local-in-time solution. Now, we shall obtain an estimate independent on $l$.

\begin{lem}
The approximate solutions obtained above  satisfies the following estimates
\eqq{  \ddt \ndk{\vl } + 2\mnit  \ndk{ D(\vl)}\leq 0,}{cf}
\eqq{ \ddt \ndk{\ol } + 2\mnit  \ndk{ \nabla \ol}\leq 0 ,}{ch}
\eqq{\ddt \ndk{\bl } + 2\mnit  \ndk{ \nabla \bl}\leq  2\nif{ \bl} \nif{ \mul} \ndk{ \nab \vl }, }{ci}
where $\mnit$ is defined by (\ref{h}).
\label{energyesti}
\end{lem}

\begin{proof}
We multiply (\ref{ca}) by $\cli$, sum over $i$ and we obtain
\[
\jd  \ddt \ndk{\vl } + (\mul D(\vl), D(\vl))=0,
\]
where we used (\ref{ad}). Applying the properties of functions $\Pst$, $\Pht$ and (\ref{h}) we get
\eqq{
\jd  \ddt \ndk{\vl } + \mnit  \ndk{ D(\vl)}\leq 0.
}{ce} Similarly, we multiply (\ref{cb}) by $\eli$ and we obtain
\[
\jd \ddt \ndk{ \ol} + (\mul \nabla \ol , \nabla \ol) = - \kd (\pht^{2}(\ol), \ol).
\]
By the properties of $\pht$ the right-hand side is non-positive thus, we obtain  (\ref{ch}). Finally, after multiplying (\ref{cc}) by $\dli$ we get
\[
\jd \ddt \ndk{ \bl} + (\mul \nabla \bl , \nabla \bl )= - (\pst(\bl)\pht(\ol), \bl) + (\mul \dvlk, \bl).
\]
We note that $\pst(\bl)\pht(\ol) \bl \geq 0$ hence, we obtain
\[
\jd \ddt \ndk{ \bl} + \mnit \ndk{ \nabla \bl} \leq (\mul \dvlk, \bl ) \leq  \nif{ \bl} \nif{ \mul} \ndk{ \nab \vl }
\]
and the proof is finished.
\end{proof}
We also need the higher order estimates.
\begin{lem}
There exist positive $\ts$ and $\cs$,  which depend on $\bmi$, $\omi$, $\oma$,  $\Om$, $\kd$, $c_{0}$, $\nsod{v_{0}}$, $\nsod{\omega_{0}}$ and  $\nsod{b_{0}}$ such that for each $l\in \NN$ the following estimate
\eqq{\| \vl, \ol, \bl  \|_{L^{\infty}(0, \ts; H^{2}(\Om))}+\| \vl, \ol, \bl  \|_{L^{2}(0,\ts;H^{3}(\Om))}      +
\| \vlt, \olt, \blt  \|_{L^{2}(0,\ts;H^{1}(\Om))}
 \leq \cs}{fa}
holds.

Furthermore, for each positive $\delta$ and compact \m{$K\subseteq \{ (a,b,c): 0<a\leq b, \hd 0<c \}$} there exists positive $\tskd$, which depends only on $\kd, \Om, \delta$ and $K$ such that if
\eqq{  \nsodk{v_{0}}+\nsodk{\om_{0}}+\nsodk{b_{0}}\leq \delta  \hd \m{ and } \hd (\omi, \oma, \bmi)\in K,}{uniformly}
then $\ts\geq \tskd$.
\label{energyhigher}
\end{lem}

Before we go to the proof of Lemma~\ref{energyhigher} we present
 its idea. First, we test the equation for approximate solution by its bi-Laplacian. Next, after integration by parts we obtain (\ref{cff}), (\ref{cg}) and (\ref{chh}). Further, we apply the lower bound for the ''diffusive coefficient'' $\mul$ (see (\ref{estimu})) and use the H\"older and Gagliardo-Nirenberg inequalities which leads to (\ref{sumvlolbl}). To estimate the $H^{2}$-norm of $\mul$ we use the properties of $\Pst$ and $\Pht$. After applying the energy estimates from Lemma~\ref{energyesti} we obtain (\ref{caloscGro}), which leads to a uniform bound of the $H^{2}$-norm of the sequence of approximate solution on the interval $(0,t^{*})$ for some positive $t^{*}$ (see (\ref{eg})). Immediately it gives a bound in $L^{2}H^{3}$. The last step is the $l$-independent estimate of the time derivative of the approximate solution.

\begin{proof}
We multiply the equality (\ref{ca}) by $\la_{i}^{2}\cli$ and sum over $i$
\[
(\vlt, \lapk \vl ) - (\vl \otimes \vl , \nabla \lapk \vl )+ (\mul D(\vl ), D(\lapk \vl ))=0.
\]
After integrating by parts we obtain
\[
(\vlt, \lapk \vl ) =  \jd \ddt \ndk{ \lap \vl },
\]
\[
(\vl \otimes \vl , \nabla \lapk \vl )=  ( \lap (\vl \otimes \vl) , \nal \vl ),
\]

\[
(\mul D(\vl ), D(\lapk \vl )) = ( \lap \mul D(\vl), \lap D(\vl))
\]
\[
+ 2 (\nab \mul \cdot \nab \dvl ,\lap \dvl) + (\mul \lap \dvl ,\lap \dvl ).
\]
Thus, we get
\[
\jd \ddt \ndk{\lap \vl } + \into \mul |\lap \dvl |^{2}dx
\]
\[
 = -( \lap (\vl \otimes \vl) , \nal \vl ) - ( \lap \mul D(\vl), \lap D(\vl)) - 2 (\nab \mul \cdot \nab \dvl ,\lap \dvl).
\]
We estimate the right-hand side
\[
|( \lap (\vl \otimes \vl) , \nal \vl ) |\leq   \nif{\vl } \nd{ \nabd\vl } \nd{ \nabt\vl } + \nck{ \nab \vl } \nd{\nabt \vl } .
\]
Proceeding analogously we obtain
\[
\jd \ddt \ndk{ \lap \vl } + \into \mul | \lap \dvl |^{2}dx
\]
\[
\leq  \nif{\vl } \nd{ \nabd\vl } \nd{ \nabt\vl } + \nck{ \nab \vl } \nd{\nabt \vl }
\]
\eqq{
+  \Big( \nd{ \lap \mul  \dvl } + 2 \nd{\nab \mul \cdot \nab \dvl } \Big)\nd{\lap \dvl}.
}{cff}
Now, we multiply the equation (\ref{cb}) by $\laf_{i}^{2} \eli$ and we obtain
\[
(\olt, \lapk \ol) - (\ol \vl, \nabla \lapk \ol  ) + \n{ \mul \nabla \ol, \nabla \lapk \ol  } = -  \kappa_{2}(\pht^{2}(\ol), \lapk \ol ).
\]
After integrating by parts we get
\[
(\olt, \lapk \ol) = \jd \ddt \ndk{ \lap \ol },
\]
\[
(\ol \vl, \nabla \lapk \ol ) = (\lap \ol  \vl , \nal \ol)+ 2(\nab \vl \nab \ol  , \nal \ol) + (\ol \lap \vl , \nal \ol),
\]
\[
\n{ \mul \nabla \ol, \nab \lapk \ol  } = \n{\lap \mul \nab \ol, \nal \ol  } + 2\n{ \nabk \ol \nab \mul , \nal \ol  }
 + \n{ \mul \nal \ol, \nal \ol },
\]
\eqns{
-(\pht^{2}(\ol), \lapk \ol )
= 2\n{ \phol \pholp \nab \ol, \nal \ol }
}
Thus, we may write
\[
\jd \ddt \ndk{\lap \ol } + \into \mul \bk{\nal \ol } dx
\]
\[
\leq \left( \nd{\lap  \ol \vl }+ \nd{ \nab \vl \nab \ol } +\nd{\ol \lap \vl }+
\nd{\lap \mul \nab \ol }  \right. \hspace{4cm}
\]
\eqq{
\left. \hspace{2cm} + 2 \nd{\nabk \ol \nab \mul  } + 2\kd\nd{\phol \pholp \nab \ol} \right) \nd{ \nal \ol}.
}{cg}
Finally, after  multiplying (\ref{cc}) by $\laf^{2}_{i}\dli$  we obtain
\[
(\blt, \lapk \bl) - (\bl \vl, \lapk \nab \bl   )   +\n{  \mul  \nab \bl, \nab \lapk \bl }
\]
\[
 = - (\pst(\bl)\pht( \ol) , \lapk \bl) + ( \mul \dvlk, \lapk \bl).
\]
We deal with the terms on the left hand-side as  earlier and for the right-hand side terms  we get
\[
-(\pst(\bl)\pht( \ol) , \lapk \bl)  = \left( \psblp   \phol \nab \bl, \nal \bl  \right)
+ \left( \psbl \pholp \nab \ol  , \nal \bl   \right),
\]
\[
( \mul \dvlk, \lapk \bl) = -(\dvlk \nab \mul  , \nal \bl ) -  (\mul \nab (\dvlk)  , \nal \bl).
\]
Therefore, we obtain the inequality
\[
\jd \ddt \ndk{\lap \bl } + \into \mul \bk{\nal \bl } dx
\]
\[
\leq  \Big( \nd{\lap \bl \vl }+ 2\nd{\nab \vl  \nab \bl } + \nd{ \bl \lap \vl } +
\nd{\lap \mul \nab \bl }  +  2\nd{\nabk \bl \nab \mul  } + \nd{\phol \psblp \nab \bl}
\]
\eqq{
 + \nd{\psbl \pholp \nab \ol} + \nd{ \nab \mul \bk{ \dvl } } +  \nd{\mul  |\dvl| |\nab \dvl|  }
 \Big) \nd{ \nal \bl}.
}{chh}
We note that
\eqq{\into  \bk{ \lap \dvl  } dx = \frac{1}{2} \into \bk{ \nabt \vl } dx. }{Korn}
Indeed,  integrating by parts yield
\[
2\into  \bk{ \lap \dvl  } dx  = \sum_{k,m}\into \bk{  \lap v^{l}_{k,x_{m}} } dx + \into \lap v^{l}_{k,x_{m}} \cdot \lap v^{l}_{m,x_{k}} dx
\]
\[
=\sum_{k,m,p,q}\into    v^{l}_{k,x_{m}x_{p}x_{p}} \cdot  v^{l}_{k,x_{m}x_{q}x_{q}}  dx + \sum_{k,m,p,q}\into \lap v^{l}_{k,x_{k}} \cdot \lap v^{l}_{m,x_{m}} dx
\]
\[
=\sum_{k,m,p,q}\into    \bk{ v^{l}_{k,x_{m}x_{p}x_{q}} } dx,
\]
where we applied the condition $\divv{ \vl} =0$ and used the tensor notation for components and derivatives. After applying (\ref{h}), (\ref{defPsi}), (\ref{defPhi}) and (\ref{cd}) we get
\eqq{
\mnit \leq \mul
}{estimu}
for each $l$ thus, (\ref{cff}) together with (\ref{Korn}) and (\ref{estimu}) give
\[
 \ddt \ndk{\lap \vl } + \mnit  \ndk{ \lap \dvl }
\]
\eqq{
 \leq  \frac{32}{\mnit} \Big( \nifk{\vl } \ndk{ \nabk\vl } + \ncc{ \nab \vl } + \ndk{  \lap \mul  \dvl  } +   \ndk{\nab  \mul \cdot \nab \dvl}  \Big).
}{cfa}
Applying Gagliardo-Nirenberg interpolation inequality
\eqq{\nif{ \nab \vl} \leq C \nd{ \nabt \vl  }^{\frac{1}{2}} \ns{  \nab \vl }^{\frac{1}{2}} }{Gaginfi}
and Sobolev embedding inequality we get
\[
\ndk{  \lap \mul  \dvl  } \leq \ndk{  \lap \mul} \nifk{  \dvl  } \leq C \nd{ \nabt \vl } \nsod{ \vl } \nsod{\mul}^{2}
,
\]
where $C$ depends only on $\Omega$. Again, by Gagliardo-Nirenberg inequality
\eqq{\nt{ \nabd \vl }   \leq  C \nd{ \nabt \vl }^{\frac{1}{2}} \nd{  \nabk \vl }^{\frac{1}{2}} }{Gagl3}
and H\"older inequality we have
\[
 \ndk{\nab  \mul \cdot \nab \dvl} \leq \nsk{\nab \mu }\ntk{\nabk \vl  } \leq C \nd{ \nabt \vl } \nsod{   \vl } \nsodk{ \mul}.
\]
Thus, applying after the Young inequality with exponents $(2,6,3)$ we get
\eqq{
\ndk{  \lap \mul  \dvl  }+ \ndk{\nab  \mul \cdot \nab \dvl} \leq \ep \ndk{ \nabt \vl }+\frac{C}{\ep } ( \nsod{ \vl }^{6} + \nsod{\mul}^{6}),
}{zGag}
where $\ep>0$ and $C$ depends only on $\Om$. Applying the above inequality and (\ref{Korn}) in (\ref{cfa}) we obtain
\eqq{
 \ddt \ndk{\nabk \vl } + \mnit  \ndk{ \nabt \vl }
 \leq  \frac{C}{\mnit} \Big( \| \vl \|^{4}_{2,2}+ (\mnit)^{-2}( \| \vl \|_{2,2}^{6}+ \| \mul \|_{2,2}^{6})  \Big),
}{zcfa}
where $C=C(\Om)$. Now, we proceed similarly with (\ref{cg}) and we obtain
\[
 \ddt \ndk{\lap \ol } + \mnit \ndk{ \nal \ol} \leq \frac{C}{\mnit}
\Big(
\nifk{ \vl} \ndk{\nabk \ol} + \nck{ \nab \vl} \nck{ \nab \ol} + \nifk{\ol} \ndk{ \nabk \vl}
\]
\eqq{+\ndk{ \lap \mul  \nab \ol} + \ndk{\nabk \ol  \nab \mul   }
+ \kappa_{2}^{2} c_{0}^{2} \nifk{\ol}  \ndk{\nab \ol} \Big),}{cgg}
where we applied (\ref{estiphi}). We repeat the reasoning leading to (\ref{zGag}) and we obtain
\[
\ndk{ \lap \mul  \nab \ol} + \ndk{\nabk \ol  \nab \mul   } \leq  \ep \ndk{ \nabt \ol }+\frac{C}{\ep } ( \nsod{ \ol }^{6} + \nsod{\mul}^{6}).
\]
Thus, the above inequality and (\ref{cgg}) give
\[
\ddt \ndk{\nabk \ol } + \mnit \ndk{ \nabt \ol}
\]
\eqq{
 \leq \frac{C}{\mnit}
\Big(
 \| \vl \|_{2,2}^{4}+ (1+\kappa_{2}^{4}c_{0}^{4} ) \| \ol \|_{2,2}^{4}+  (\mnit)^{-2}( \| \ol \|_{2,2}^{6}+ \| \mul \|_{2,2}^{6})  \Big),}{zcgg}
where $C=C(\Om)$. Further, from (\ref{chh}) we get
\[
\ddt \ndk{\lap \bl } + \mnit \ndk{ \nal \bl } \leq \frac{C}{\mnit} \Big(  \nifk{ \vl } \ndk{ \nabk \bl } + \nck{ \nab \vl} \nck{ \nab \bl } + \nifk{\bl} \ndk{ \nabk \vl}
\]
\[
 +
\ndk{\nabk \mul \nab \bl }  +  \ndk{\nabk \bl \nab \mul  } + c_{0}^{2} \nifk{ \ol} \ndk{ \nab \bl}
\]
\[
 + c_{0}^{2} \nifk{ \bl} \ndk{ \nab \ol} + \ndk{ \nab \mul |\dvl|^{2} } +  \ndk{\mul \nab(|\dvl|^{2}) }\Big),
\]
where we applied (\ref{estipsi}) and (\ref{estiphi}). Applying integrating by parts and Sobolev embedding theorem we get
\[
\ddt \ndk{\nabk \bl } + \mnit \ndk{ \nabt \bl } \leq \frac{C}{\mnit} \Big(  \nsod{ \vl }^{4}+ \nsod{  \bl }^{4} +
\ndk{\nabk \mul \nab \bl }  +  \ndk{\nabk \bl \nab \mul  }
\]
\eqq{
 + c_{0}^{4} \nsod{ \ol}^{4} +\nsod{\mul }^{6} +\nsod{\vl }^{6} +  \ntk{\nabk \vl   } \nsod{ \mul }^{2} \nsod{ \vl }^{2} \Big),
}{nabzGag}
Applying again the Gagliardo-Nirenberg inequality and Young inequality we get
\[
\ndk{\nabk \mul \nab \bl }  +  \ndk{\nabk \bl \nab \mul  } \leq \ep \| \nabt \bl \|_{2}^{2} +  \frac{C}{\ep}( \| \bl \|_{2,2}^{6}+ \| \mul \|_{2,2}^{6}).
\]
From (\ref{Gagl3}) we get
\[
\ntk{\nabk \vl   }  \nsod{ \vl }^{2} \nsod{ \mul }^{2}\leq C\nd{\nabt \vl   }\nsod{ \vl }^{3} \nsod{ \mul }^{2} \leq \ep \ndk{\nabt \vl   }+\frac{C}{\ep}( \nsod{ \vl }^{10}+ \nsod{ \mul }^{10}).
\]
hence, from (\ref{nabzGag}) we obtain the following  estimate
\[
\ddt \ndk{\nabk \bl } + \mnit \ndk{ \nabt \bl } \leq \frac{C}{\mnit} \Big(  \nsod{ \vl }^{4}+ \nsod{  \bl }^{4} + c_{0}^{4} \nsod{ \ol}^{4} +\nsod{\mul }^{6} +\nsod{\vl }^{6} \Big)
\]
\eqq{ +\frac{C}{(\mnit)^{3}}  \Big(  \nsod{\bl}^{6}+ \nsod{\mul}^{6}+ \nsod{\vl}^{10}+ \nsod{\mul}^{10}    \Big)  +\frac{\mnit}{2}  \ndk{\nabt \vl   }, }{nablzGag}
where $C=C(\Om)$. We sum the inequalities (\ref{zcfa}), (\ref{zcgg}), (\ref{nablzGag}) and we obtain
\[
\ddt \Big( \ndk{\nabk \vl } +\ndk{\nabk \ol } + \ndk{\nabk \bl }  \Big) + \mnit  \Big( \ndk{\nabt \vl } +\ndk{\nabt \ol } + \ndk{\nabt \bl }  \Big)
\]
\[
\leq \frac{C}{\mnit} \Big(  \nsod{ \vl }^{4}+ \nsod{  \bl }^{4} + (1+c_{0}^{4}+c_{0}^{4}\kappa_{2}^{4}) \nsod{ \ol}^{4} +\nsod{\mul }^{6} +\nsod{\vl }^{6} \Big)
\]
\eqq{+\frac{C}{(\mnit)^{3}} \Big(  \nsod{ \vl }^{6}+ \nsod{  \bl }^{6} +  \nsod{ \ol}^{6} +\nsod{\mul }^{6} +\nsod{\vl }^{10}+\nsod{\mul }^{10} \Big) }{sumvlolbl_old}
for some $C$, which depends only on $\Om$. We note that
\eqq{
\mnit = \frac{1}{4} \frac{\bmi}{\oma}(1+\kd \oma t )^{1-\frac{1}{\kd}}
}{mutmineq}
hence, we have
\[
\ddt \Big( \ndk{\nabk \vl } +\ndk{\nabk \ol } + \ndk{\nabk \bl }  \Big) + \mnit  \Big( \ndk{\nabt \vl } +\ndk{\nabt \ol } + \ndk{\nabt \bl }  \Big)
\]
\eqq{
\leq C \n{\frac{\oma}{\bmi}+ \n{\frac{\oma}{\bmi}}^{3} }\n{1 + \kd \oma t}^{\beta } \Big( 1 +  \nsod{  \bl }^{6} +  \nsod{ \ol}^{6} +\nsod{\mul }^{10} +\nsod{\vl }^{10} \Big),
}{sumvlolbl}
where $\beta=  \max \{  \frac{1}{\kd}-1,  \frac{3}{\kd}-3 \}$ and $C$ depends only on $\Om$, $c_{0}$ and $\kd$.

Now, we shall estimate $\mul$ in terms of  $\ol$ and $\bl$. Firstly, we note that from (\ref{defPsi}) and (\ref{defPhi}) we have
\eqq{\Psbl \leq \max\{\jd \bmt , \bl   \}, \hd \hd \Phol \geq \jd \omt.}{da}
Hence, by definition (\ref{cd}) we get
\eqq{0<\mul \leq 2 \omtm \max\{\bmt, \bl  \}\leq c_1(\Omega) \frac{1}{\omi}  \n{1 + \kd \omi t}  \n{ \bmi +	|\bl| } ,}{mulestibe}
where $c_{1}$ depends only on $\Omega$. Thus, we obtain
\eqq{
\nd{\mul} \leq c_{1}\frac{1}{\omi} \n{1 + \kd \omi t} (\bmi + \nd{\bl}).    }{mulifty}
Now, we have to estimate the derivatives of $\mul$. Direct calculation gives
\[
| \nabk \mul| = \be{\nabk \n{ \Psbl \cdot (\Phol)^{-1}  }} \leq (\Phol)^{-1} \be{\nabk (\Psbl)}
\]
\[
+2(\Phol)^{-2} \be{\nab (\Psbl)} \be{\nab( \Phol) }
\]
\eqq{
+2\Psbl (\Phol)^{-3} \be{\nab( \Phol) }^2 + \Psbl (\Phol)^{-2} \be{\nabk( \Phol) } .
}{muleq}
Using (\ref{estiPsi}) and (\ref{estiPhi}) we may estimate the derivatives
\eqq{\be{\nab (\Psbl ) } \leq c_{0} \be{ \nab \bl }, \hd \hd  \be{\nab (\Phol ) } \leq c_{0} \be{ \nab \ol }, }{db}
\eqq{\begin{array}{l}\vspace{0.2cm} \be{\nabk (\Psbl ) } \leq c_{0}\bmtm   \bk{ \nab \bl } + c_{0} \be{ \nabk \bl }, \hd \hd  \\ \be{\nabk (\Phol ) } \leq c_{0}\omtm   \bk{ \nab \ol } + c_{0} \be{ \nabk \ol }. \end{array} }{dc}
If we apply estimates (\ref{da}), (\ref{db}) and (\ref{dc}) in (\ref{muleq}) then we obtain
\[
\be{\nabk \mul } \leq c_{2} Q_{1} \n{1+ \kd \oma t }^{\max \{ 3, 1 + \frac{1}{\kd } \} }  \left[ \be{\nab \bl }^{2}  + \be{\nabk \bl } +|\bl | \be{  \nab \ol}^{2}
\right.
\]
\eqq{
\left.+ \be{\nab \bl }+ \be{\nab \ol }+ \be{\nab \ol }^{2} + \be{\bl \nabk \ol} + \be{\nabk \ol}  \right]  }{mulesta}
where  $c_{2}$ depends only on $c_{0}$ and $Q_{1}=\frac{\bmi}{\omi}\n{1+\bmi^{-3}+\omi^{-3}  } $. Thus, we obtain
\[
\nd{\nabk \mul }\leq c_{2}Q_{1}\n{1+ \kd \oma t }^{\max \{ 3,1 + \frac{1}{\kd } \}  }  \left[ \nck{\nab\bl  } +  \nd{\nabk \bl }  \right.
\]
\eqq{  \left. +  \nif{\bl} \nck{\nab \ol } + \nck{ \nab \ol} + \nd{ \nabk \ol }  + \nif{ \bl } \nd{\nabk \ol }  \right]  .     }{estimuldiff}
If we take into account (\ref{mulifty}) then we get
\eqq{\nsod{ \mul} \leq c_{3} Q_{1}  \n{1+ \kd \oma t }^{\max \{ 3,1 + \frac{1}{\kd } \} } \left(  \nsodt{ \bl} + \nsodt{ \ol}+1 \right), }{ea}
where $c_{3}=c_{3}( c_{0}, \Omega )$. Applying the above estimate in (\ref{sumvlolbl}) we obtain
\[
\ddt \Big( \ndk{\nabk \vl } +\ndk{\nabk \ol } + \ndk{\nabk \bl }  \Big) + \mnit  \Big( \ndk{\nabt \vl } +\ndk{\nabt \ol } + \ndk{\nabt \bl }  \Big)
\]
\eqq{
\le C Q_{2} \n{1+ \kd \oma t }^{\bar{\beta} } \Big( 1+ \nsod{ \vl }^{2}+ \nsod{  \bl }^{2} +  \nsod{ \ol}^{2}  \Big)^{15}, }{sumvlolblbezmu}
where
\[
Q_{2}= \left[ 1+\n{\frac{\oma}{\bmi}}^{3} \right] \left[ \frac{\bmi}{\omi}(1+\bmi^{-3} +\omi^{-3} )^{10} +1\right],  \hd \bar{\beta}= 10\max\{1+\frac{1}{\kd},3\}+\beta
\]
and $C$ depends only on $\Om$, $c_{0}$ and $\kd$.
If we take into account the estimates (\ref{cf})-(\ref{ci}) then we have
\[
\ddt \Big( \nsodk{ \vl } +\nsodk{ \ol } + \nsodk{ \bl }  \Big) + \mnit  \Big( \nsotk{ \vl } +\nsotk{ \ol } + \nsotk{ \bl }  \Big)
\]
\eqq{\leq C Q_{3}\n{1+ \kd \oma t }^{\bar{\beta} } \Big( 1+ \nsod{ \vl }^{2}+ \nsod{  \bl }^{2} +  \nsod{ \ol}^{2}  \Big)^{15}, }{caloscGro}
where $C=C( c_{0}, \Om, \kappa_{2} )$ and $Q_{3}=Q_{1}^{2}+Q_{2}+1$. If we divide both sides by the last term and next integrate with respect time variable then we get
\[
\Big( 1+ \nsod{ \vl(t) }^{2}+ \nsod{  \bl(t) }^{2} +  \nsod{ \ol(t)}^{2}  \Big)^{-14}
\]
\[
\geq \Big( 1+ \nsod{ \vl(0) }^{2}+ \nsod{  \bl(0) }^{2} +  \nsod{ \ol(0)}^{2}  \Big)^{-14} - \frac{14CQ_{3}}{(\bar{\beta}+1) \kd \oma } \n{ (1+\kd \oma t )^{\bar{\beta}+1} -1}
\]
\eqq{
\geq \Big( 1+ \nsod{ v_{0} }^{2}+ \nsod{  b_{0} }^{2} +  \nsod{ \omega_{0}}^{2}  \Big)^{-14} - \frac{14CQ_{3}}{(\bar{\beta}+1) \kd \oma } \n{ (1+\kd \oma t )^{\bar{\beta}+1} -1},
}{dodefts}
where the last estimate is a consequence of Bessel inequality. Now, we define time $\ts$ as the unique solution of  the equality
\eqq{\Big( 1+ \nsod{ v_{0} }^{2}+ \nsod{  b_{0} }^{2} +  \nsod{ \omega_{0}}^{2}  \Big)^{-14} = \frac{15CQ_{3}}{(\bar{\beta}+1) \kd \oma } \n{ (1+\kd \oma \ts )^{\bar{\beta}+1} -1}.}{defts}
We note that $\ts $ is positive and depends on $\nsod{ v_{0} }^{2}+ \nsod{  b_{0} }^{2} +  \nsod{ \omega_{0}}^{2}$, $\kd$, $\Om$, $c_{0}$, $\omi$, $\oma$ and $\bmi$.  It is evident that $\ts$ is decreasing function of $\nsod{ v_{0} }^{2}+ \nsod{  b_{0} }^{2} +  \nsod{ \omega_{0}}^{2} $. Moreover, for any $\delta>0$ and compact $K\subseteq \{(a,b,c): \hd 0<a\leq b, \hd 0<c \}$ there exists $\tskd >0$ such that $\ts \geq \tskd$ for any initial data satisfying $\nsod{ v_{0} }^{2}+ \nsod{  b_{0} }^{2} +  \nsod{ \omega_{0}}^{2} \leq \delta$ and $(\omi, \oma, \bmi)\in K$. From (\ref{defts})  we deduce that $\tskd$ depends only on $\delta$, $K$, $\Om$ $\kd$ and $ c_{0}$.

From (\ref{dodefts}) and (\ref{defts}) we have
\[
\Big( 1+ \nsod{ \vl(t) }^{2}+ \nsod{  \bl(t) }^{2} +  \nsod{ \ol(t)}^{2}  \Big)^{-14} \geq \frac{CQ_{3}}{(\bar{\beta}+1) \kd \oma } \n{ (1+\kd \oma t )^{\bar{\beta}+1} -1}
\]
for $t\in [0,\ts]$ hence,
\eqq{
\nsod{ \vl(t) }^{2}+ \nsod{  \bl(t) }^{2} +  \nsod{ \ol(t)}^{2}   \leq \left[ \frac{CQ_{3}}{(\bar{\beta}+1) \kd \oma } \n{ (1+\kd \oma \ts )^{\bar{\beta}+1} -1}\right]^{-\frac{1}{14}}
}{defCs}
for $t\in [0,\ts]$. In particular, there exists $\cgs=\cgs(\ts)$ such that
\eqq{\| \vl \|_{L^{\infty}(0, \ts; \Vkdd)}+\| \ol \|_{L^{\infty}(0, \ts; \Vd)}+ \| \bl \|_{L^{\infty}(0, \ts; \Vd)} \leq \cgs}{eg}
uniformly with respect to $l\in \NN$. Next, from (\ref{mutmineq}), (\ref{caloscGro}) and  (\ref{eg})  we get the bound
\eqq{\| \vl \|_{L^{2}(0, \ts; \Vkdt)}+\| \ol \|_{L^{2}(0, \ts; \Vt)}+ \| \bl \|_{L^{2}(0, \ts; \Vt)} \leq \cs,}{gh}
where $\cs $ depends on $\ts, \kd, $ $\bmi$, $\oma$ and $\cgs$.
It remains to show the estimate of time derivative of solution. We do this by multiplying the equality (\ref{ca}) by $\clit$ and after summing it over $i$ we get
\[
(\vlt, \vlt ) - (\vl \otimes \vl , \nabla \vlt )+ (\mul D(\vl ), D(\vlt ))=0.
\]
Thus, by after integration by parts and applying H\"older inequality we have
\[
\ndk{\vlt} \le \nd{\divv (\vl \otimes \vl)} \nd{\vlt} + \nd{\nab  \n{\mul D(\vl )}} \nd{\vlt}.
\]
By applying Young inequality we get
\[
\ndk{\vlt} \le 2\ndk{\divv (\vl \otimes \vl)} + 2\ndk{\nab \n{\mul D(\vl )}} .
\]
Next,  H\"older inequality gives us
\[
\ndk{\vlt} \le C \Big(
  \nck{\nab \vl} \nck{\vl}
+ \nck{\nab \mul} \nck{ D(\vl )}
+ \nif{\mul}^2 \ndk{\nab D(\vl )}
\Big).
\]
Finally, Sobolev embedding theorem leads us to the following inequality
\[
\ndk{\vlt} \le C \Big(
  \nsod{\vl}^4
+ \nsodk{\mul} \nsodk{\vl}
\Big),
\]
where $C$ depends only on $\Om$. If we apply (\ref{ea}) and (\ref{eg}) then we get
\eqq{\| \vlt \|_{L^{\infty}(0,\ts;L^{2}(\Om))} \leq \cs, }{gi}
where $\cs $ depends on $\Om, c_{0}, \ts, \kd, $ $\bmi$, $\oma$ and $\cgs$.

Now, we shall consider (\ref{cb}).  Proceeding  as earlier we get
\[
\ndk{\olt} \leq  4\ndk{ \nab \ol \cdot \vl }+ 4\ndk{ \nab ( \mul \nab \ol)}+ 4\kappa_{2}\ndk{ \pht^{2} (\ol)}
\]
\[
\leq 4 \nifk{ \vl } \ndk{ \nab \ol } + 8\nck{ \nab \mul } \nck{ \nab \ol } + 8\nifk{ \mul } \ndk{\nabk \ol}+ 4\kappa_{2}\ncc{\ol},
\]
where we applied (\ref{estiphi}). Thus, using (\ref{ea}) and (\ref{eg}) we get
\eqq{\| \olt \|_{L^{\infty}(0,\ts;L^{2}(\Om))} \leq \cs,  }{gl}
where $\cs$ is as earlier. It remains to deal with  (\ref{cc}). In similar way we obtain
\[
\ndk{ \blt} \leq 4\ndk{ \nab \bl \vl } + 4\ndk{ \nab (\mul \nab \bl )}+ 4\ndk{\psbl \phol } + 4\ndk{ \mul | \dvl|^{2}}
\]
\[
\leq 4\ndk{\nab  \bl } \nifk{ \vl} + 8\nck{ \nab \mul} \nck{ \nab \bl } + 8\nifk{ \mul }\ndk{ \nabk \bl }+ 4\nifk{ \bl }\ndk{ \ol} + 4\nifk{ \mul }\ncc{ \nab \vl}.
\]
Applying again (\ref{ea}) and (\ref{eg}) we obtain
\eqq{\| \blt \|_{L^{\infty}(0,\ts;L^{2}(\Om))} \leq \cs,  }{gm}
where $\cs $ depends on $\Om, c_{0}, \ts, \kd, $ $\bmi$, $\oma$ and $\cgs$.

Now, we prove the higher order estimates for time derivative of approximate solution. Firstly, we multiply the equality (\ref{ca}) by $-\la_{i} \clit$ and sum over $i$
\[
(\vlt, -\lap \vlt ) + (\vl \otimes \vl , \nabla \lap \vlt )- (\mul D(\vl ), D(\lap \vlt ))=0.
\]
After integration by parts we get
\[
\ndk{\nab \vlt }
=
- \n{\lap \n{ \vl \otimes \vl }, \nabla \vlt }
+ \n{\lap \n{\mul D(\vl )}, D(\vlt )}.
\]
If we apply H\"older and Young inequalities, then we get
\[
\ndk{\nab \vlt }
\le
 2\nd{\lap \n{ \vl \otimes \vl }}^2
+ \nd{\lap \n{\mul D(\vl )}}^2,
\]
where we used the equality $ 2\ndk{D(\vlt )} = \ndk{\nab \vlt} $. We estimate further
\[
\ndk{\nab \vlt }
\le
  8\nif{\vl}^2\nd{\nabk \vl}^2
+ 8\ncc{\nab \vl} + 4\nif{\mul}^2 \nd{\lap D(\vl )}^2
\]
\[
+ 16\nt{\nab \mul}^2 \ns{\nab D(\vl )}^2
+ 4\nd{\lap \mul}^2 \nif{D(\vl )}^2.
\]
Using Sobolev embedding we obtain
\[
\ndk{\nab \vlt }
\le  C \Big(
  \nsod{ \vl }^4
+ \nsodk{\mul} \nsodk{\vl}
+ \nsodk{\mul} \nsot{\vl}^2.
\Big),
\]
where $C$ depends only on $\Om$. Applying  (\ref{ea}),  (\ref{eg}) and (\ref{gh}) we get
\eqnsl{
\| \nab \vlt \|_{L^{2}(0, \ts; L^{2}(\Om)} \le \cs,
}{vtEst_2}
where $\cs $ depends on $c_{0}, \Om, \ts, \kd, $ $\bmi$, $\oma$ and $\cgs$. Proceeding analogously we get
\eqnsl{
\| \nab \olt \|_{L^{2}(0, \ts; L^{2}(\Om))} \le \cs.
}{oltEst_1}
It remains to estimate $\nab \blt$. If we  multiply the equality (\ref{cc}) by $-\laf_{i} \dlit$ and sum over $i$, then we get
\[
(\blt, - \lap \blt) + (\bl \vl, \nab \lap \blt   )
- \n{  \mul  \nabla \bl, \nabla \lap \blt  }
\]
\[
=
   (\pst(\bl)\pht( \ol) , \lap \blt)
- ( \mul \dvlk, \lap \blt).
\]
Integrating by parts and  H\"older inequality lead to
\[
\ndk{\nab \blt}
 \le
  \nd{\lap \n{\bl \vl}} \nd{\nab \blt}
+ \nd{ \lap \n{\mul  \nabla \bl}} \nd{\nab \blt}
\]
\[
+   \nd{\nab \n{\pst(\bl)\pht( \ol)}} \nd{\nab \blt }
+ \nd{\nab \n{ \mul \dvlk}} \nd{\nab \blt}.
\]
After applying Young inequality  we get
\[
\ndk{\nab \blt}
\]
\[
 \le
  4 \nd{\lap \n{\bl \vl}}^2
+ 4 \nd{ \lap \n{\mul  \nabla \bl}}^2
+ 4 \nd{\nab \n{\pst(\bl)\pht( \ol)}}^2
+ 4 \nd{\nab \n{ \mul \dvlk}}^2.
\]
Using H\"older inequality we obtain
\eqnsl{
\ndk{\nab \blt}
& \le
  16 \nd{\lap \bl}^2 \nif{\vl}^2
+ 32 \nc{\nab \bl}^2 \nc{\nab \vl}^2
+ 16\nif{\bl}^2 \nd{\nabk \vl}^2 \\
&
+ 16 \nd{ \lap \mul}^2  \nif{\nabla \bl}^2
+ 32 \nc{ \nab \mul}^2  \nc{\nabk \bl}^2
+ 16 \nif{\mul}^2  \nd{\nal \bl}^2  \\
&
+ 8 \nd{\nab ( \pst(\bl))}^2\nif{\pht( \ol)}^2
+ 8 \nif{\pst(\bl)}^2\nd{\nab ( \pht( \ol))}^2 \\
&
+ 8 \ns{\nab \mul}^2 \ns{\dvl}^4
+ 16 \nif{\mul}^2 \nt{\dvl}^2\ns{\nab \dvl}^2.
}{estB_1__}
After applying (\ref{estipsi}) and (\ref{estiphi}) we get $\nif{\pst(\bl)} \le \nif{\bl}$, \hd $\nif{\pst(\ol)} \le \nif{\ol}$ and
\[
\nd{\nab (\pht(\ol)) } = \nd{\pholp\nab \ol} \le c_0  \nd{\nab \ol},
\]
\[
\nd{\nab (\pst(\ol))} = \nd{\psblp\nab \ol} \le c_0  \nd{\nab \bl}.
\]
Using these inequalities in  (\ref{estB_1__}) we obtain
\[
\ndk{\nab \blt}
 \le C \Big(
  \nsodk{\bl} \nsodk{\vl}
+ \nsodk{ \mul}  \nsotk{\bl}
+ \ndk{\nab \bl} \nsodk{\ol}
+ \ndk{\nab \ol} \nsodk{\bl}
\]
\[
+ \nsodk{\mul} \nsod{\vl}^4
+ \nsodk{\mul} \nsodk{\vl} \nsotk{\vl}
\Big),
\]
where $ C = C(\Omega,c_0)$. Finally,  from (\ref{ea}), (\ref{eg}) and  (\ref{gh}) we obtain
\eqnsl{
\| \nab \blt \|_{L^{2}(0, \ts; L^{2}(\Om))} \le \cs ,
}{btEst_4}
where $\cs $ depends on $c_{0}, \Om, \ts, \kd, $ $\bmi$, $\oma$ and $\cgs$. The estimates (\ref{eg})-(\ref{gm}), (\ref{vtEst_2}), (\ref{oltEst_1}) and (\ref{btEst_4}) give (\ref{fa}) and the proof of lemma~\ref{energyhigher} is finished.

\end{proof}

Now, we draw the idea of the remain part of the proof of theorem~\ref{main}. From the $l$-independent estimate (\ref{fa}) we deduce the existence of a subsequence, which converges weakly in some spaces (see (\ref{fb})-(\ref{fd})). Next, by applying Aubin-Lions lemma we get strong convergence of the approximate solution, see (\ref{fe}), (\ref{ff}). Further, we prove the convergence of ''diffusive coefficient'' $\mul$ (\ref{ga}), which allows us to take the limit in the approximate problem. As a result, we obtain (\ref{gb})-(\ref{ccg}). In the last step we prove a series of inequalities (\ref{gc})-(\ref{le}), (\ref{ge}), (\ref{gf}), which show that the truncated problem is in fact the original one.

Having the estimate (\ref{fa}) from lemma~\ref{energyhigher} we may apply weak-compactness argument to the sequence of approximate solutions and we obtain a subsequence (still  numerated by superscript $l$) weakly convergent in appropriate spaces. To be more precise, there exist $v$, $\om$ and $b$ such that
\[
v \in L^{2}(0,\ts;\Vkdt)\cap L^{\infty}(0,\ts;\Vkdd), \hd  \vt \in L^{2}(0,\ts;H^{1}(\Om))
\]
\[
\om, b \in L^{2}(0,\ts;\Vt)\cap L^{\infty}(0,\ts;\Vd), \hd  \ot , \bt \in L^{2}(0,\ts;H^{1}(\Om))
\]
and
\eqq{\vl \sla v \m{ in } L^{2}(0,\ts;\Vkdt), \hd \hd  \vl \slg v \m{ in }   L^{\infty}(0,\ts;\Vkdd), \hd \hd \vlt \sla \vt  \m{ in }  L^{2}(0,\ts;H^{1}(\Om)), }{fb}
\eqq{ (\ol, \bl ) \sla (\om, b) \m{ in }  L^{2}(0,\ts;\Vt), \hd (\ol, \bl ) \slg (\om, b ) \m{ in } L^{\infty}(0,\ts;\Vd),   }{fc}
\eqq{(\olt, \blt) \sla (\ot , \bt ) \m{ in } L^{2}(0,\ts;H^{1}(\Om)).}{fd}
Thus, by the Aubin-Lions lemma there exists a subsequence (again denoted by $l$) such that
\eqq{(\vl, \ol, \bl ) \zb (v, \om , b ) \hd \m{ in } \hd L^{2}(0,\ts;H^{s}(\Om)) \hd \m{ for } \hd s< 3,}{fe}
and
\eqq{(\vl, \ol, \bl ) \zb (v, \om , b ) \hd \m{ in } \hd C([0,\ts];H^{q}(\Om)) \hd \m{ for } \hd q < 2.}{ff}
Now, we characterize the limits of nonlinear terms. Firstly, we note that for fixed $(x,t)$ we may write
\[
\Pst(\bl (x,t)) - \Pst(b(x, t) )= \int_{0}^{1} \frac{d}{ds} \left[ \Pst \left(s \bl (x,t) +(1-s) b(x, t) \right)  \right] ds
\]
\[
= \int_{0}^{1}  \Pst' (s \bl (x,t )  +(1-s) b(x, t) )  ds \cdot [\bl (x,t)- b(x,t)].
\]
Taking into account (\ref{estiPsi}) we get
\[
|\Pst(\bl (x,t)) - \Pst(b(x, t) )|\leq c_{0} |\bl (x,t)- b(x,t)|.
\]
Similarly we obtain
\[
|\Pht(\ol (x,t)) - \Pht(\om(x, t) )|\leq c_{0} |\ol (x,t)- \om(x,t)|.
\]
and
\[
|\Pht (b(x,t))| \leq c_{0} (|b(x,t)|+ \bmt).
\]
Therefore, applying (\ref{defPhi})  we obtain
\[
\left| \frac{\Psbl}{\Phol}- \frac{\Pst( b)}{ \Pht( \om) }\right|\leq 4\omtd \left[ |\Pht( \om) | | \Psbl - \Pst (b)|+ |\Pst(b )| |\Pht ( \om) - \Phol| \right]
\]
\[
\leq 4 \omtd \left[ 2 \oma |\bl - b |+ c_{0} (|b|+ \bmt)|\om - \ol|  \right].
\]
From (\ref{ff}) and the above estimate  we have
\eqq{\mul \zb \mup\equiv \frac{\Pst( b)}{ \Pht( \om)}  \hd \m{ uniformly on } \hd \overline{\Om } \times [0,\ts]. }{ga}
Now, we shall take the limit $l\rightarrow \infty$ in the system (\ref{ca})-(\ref{cc}). First, we multiply (\ref{ca}) by $a_{i}$ and sum over $i \in \{1,\dots , l \}$ and after integrating with respect time variable we get
\[
\itt (\vlt, w)dt - \itt (\vl\otimes \vl,\nabla w)dt  +   \itt  \n{ \mul  D(\vl), D(w)}dt  = 0,
\]
where $w= \sum\limits_{i=1}^{l}a_{i}w_{i}$ and $t\in (0,\ts)$. We note that from (\ref{ff}) we have for some $ \lambda > 0$
\eqq{(\vl, \ol, \bl ) \zb (v, \om , b ) \hd \m{ in } \hd C([0,\ts];C^{0,\lambda}(\overline{\Om}))}{ffa}
hence, (\ref{fd}), (\ref{ff}) and (\ref{ga}) imply that
\[
\itt (\vt, w)dt - \itt (v\otimes v,\nabla w)dt  +   \itt  \n{ \mup  D(v), D(w)}dt  = 0
\]
for $t\in (0,\ts)$ and $w= \sum\limits_{i=1}^{l}a_{i}w_{i}$. By density, the above identity holds for $w\in \Vkdj$. As a consequence, we obtain
\[
\int_{t_{1}}^{t_{2}} (\vt, w)dt - \int_{t_{1}}^{t_{2}} (v\otimes v,\nabla w)dt  +   \int_{t_{1}}^{t_{2}} \n{ \mup  D(v), D(w)}dt  = 0
\]
for $0<t_{1}<t_{2}<\ts$ and  $w\in \Vkdj$. After dividing both sides by $|t_{2}-t_{1}|$ and taking the limit $t_{2}\rightarrow t_{1}$ we get
\eqq{(\vt , w) -  (v\otimes v,\nabla w)  +     \n{ \mup  D(v), D(w)}   =0 \hd \m{ for } \hd w\in \Vkdj}{gb}
for a.a. $t\in (0,\ts)$. Further, we have
\[
\psbl \longrightarrow \pst( b), \hd \hd \phol \zb \pht (\ol) \hd \m{ uniformly on } \hd \overline{\Om } \times [0,\ts]
\]
thus, using  (\ref{cb}) and (\ref{cc}) and  arguing as earlier we obtain
\eqq{ (\ot , z) - (\om v, \nabla z  ) + \n{ \mup \nabla \om, \nabla z  } = -  \kappa_{2}(\pht^{2}(\om), z ) \hd \m{ for } \hd z\in \Vj,  }{cbg}
\eqq{(\bt , q) - (b v, \nab q   )   +\n{  \mup  \nabla b, \nabla q  }  = - (\pst(b)\pht( \om) , q) + ( \mup \bk{D(v)}, q)  \m{ for }  q\in \Vj   }{ccg}
for a.a. $t\in (0,\ts)$.

Now, we shall prove the bounds for $b$ and $\om$. The proof is similar to one found in \cite{MiNa}. We denote by $\bp $ ($\bm$) the positive (negative resp.) part of $b$. Then $b=\bp + \bm $. We shall show that
\eqq{b\geq 0 \hd \m{ in }  \hd \overline{\Om } \times [0,\ts].}{gc}
For this purpose we test the equation (\ref{ccg}) by $\bm$ and we obtain
\[
(\bt , \bm ) - (b v, \nab \bm    )   +\n{  \mup  \nabla b, \nabla \bm   }  = - (\pst(b)\pht( \om) , \bm ) + ( \mup \bk{D(v)}, \bm ).
\]
We note that from (\ref{ga}) we have $0\leq \mup$ and by (\ref{defpsi}) we obtain $\pst(b)\bm\equiv 0$ thus, we get
\[
({\partial_{t} \bm}, \bm ) - (\bm v, \nab \bm    )   +\n{  \mup  \nabla \bm, \nabla \bm   }  \leq 0
\]
and then
\[
\ddt \ndk{ \bm } \leq 0.
\]
By the assumption (\ref{f}) the negative part of initial value of $b$ is zero hence, $\bm \equiv 0$ and we obtained (\ref{gc}).

Proceeding similarly we introduce the decomposition $\om = \op + \ou$ and test the equation (\ref{cbg}) by $\ou$
\[
(\ot, \ou) - (\om v, \nabla \ou  ) + \n{ \mup \nabla \om, \nabla \ou  } = -  (\pht^{2}(\om), \ou ).
\]
We note that by  $(\ref{defphi})$ the right-hand side of the above equality vanishes thus, we get $\ddt \ndk{ \ou } \leq 0$ and by assumption (\ref{g})
\eqq{\om\geq 0 \hd \m{ in }  \hd \overline{\Om } \times [0,\ts].}{gd}
Now, we shall prove that
\eqq{\om(x,t)\geq \omitt \hd \m{ for } \hd (x,t)\in \overline{\Om } \times [0,\ts].   }{le}
We test the equation (\ref{cbg}) by $(\om - \omt)_{-}$ and we obtain
\[
(\ot, (\om - \omt)_{-})  - (\om v, \nabla (\om - \omt)_{-}  ) + \n{ \mup \nab \om, \nab \n{\om - \omt}_{-}  }
\]
\eqq{
 = -  \kd (\pht^{2}(\om), (\om - \omt)_{-} ).
}{rra}
Using (\ref{h}) we get
\[
(\ot, (\om - \omt)_{-}) = \jd \ddt \ndk{(\om - \omt)_{-}}  - \kd \n{ (\omt)^{2} ,(\om - \omt)_{-}}
\]
hence, using inequality $ 0\leq \mup $  and $\divv v=0$ in (\ref{rra}) we obtain
\[
\jd  \ddt \ndk{(\om - \omt)_{-}} - \kd \n{ (\omt)^{2} ,(\om - \omt)_{-}} \leq -  \kd (\pht^{2}(\om), (\om - \omt)_{-} ) .
\]
We write the above inequality the form
\[
\jd  \ddt \ndk{(\om - \omt)_{-}}  \leq   -\kd ((  \pht(\om)-\omt)(\pht(\om)+\omt), (\om - \omt)_{-} ).
\]
We note that  $-\kd ((\pht(\om)+\omt), (\om - \omt)_{-} )$ is nonnegative thus, using (\ref{estiphi}) we get $\pht(\om)\leq \om$ we have
\[
\jd  \ddt \ndk{(\om - \omt)_{-}}  \leq   -\kd ((  \om-\omt)(\pht(\om)+\omt), (\om - \omt)_{-} )
\]
\[
=-\kd \big( (\pht(\om)+\omt), \bk{(\om - \omt)_{-}} \big)\leq 0.
\]
Therefore, we obtain  $\ddt \ndk{ (\om - \omt)_{-} }\leq 0$ and by (\ref{g}) we get (\ref{le}).\\
Now, we shall prove that
\eqq{\om(x,t)\leq \omat \hd \m{ for } \hd (x,t)\in \overline{\Om } \times [0,\ts].   }{ge}
Indeed, firstly  we note that from (\ref{h}), (\ref{defphi}) and (\ref{le}) we have
\eqq{\pht(\om)=\om}{newa}
hence, if  we test the equation (\ref{cbg}) by $(\om - \ommt)_{+}$ then  we obtain
\[
(\ot, (\om - \ommt)_{+})  - (\om v, \nabla (\om - \ommt)_{+}  ) + \n{ \mup \nab \om, \nab \n{\om - \ommt}_{+}  }
\]
\[
= -  \kd (\om^{2}, (\om - \ommt)_{+} ).
\]
Proceeding as earlier, we get
\[
\jd \ddt \ndk{(\om - \ommt)_{+}}  -\kd \n{ (\ommt)^{2} ,(\om - \ommt)_{+}  }
 \le -  \kd(\om^{2}, (\om - \ommt)_{+} ).
\]
and
\[
\jd  \ddt \ndk{(\om - \ommt)_{+}}  \leq   -\kd ((  \om-\ommt)(\om+\ommt), (\om - \ommt)_{+} )
\]
\[
= - \kd ((\om+\ommt), |(\om-\ommt)_{+}|^{2}  )
\]
hence, we obtain
\eqns{
\jd & \ddt \ndk{(\om - \ommt)_{+}}  \le 0.
}
By (\ref{g}) we get (\ref{ge}).
We  shall prove that
\eqq{b(x,t) \geq \bmt \hd \m{ for } \hd (x,t)\in \overline{\Om } \times [0,\ts].  }{gf}
For this purpose we test the equation (\ref{ccg}) by $(b - \bmt)_{-}$. Then we get
\[
(\bt, (b - \bmt)_{-}) - (b v, \nab ((b - \bmt)_{-} )  )   +\n{  \mup  \nab b, \nab ((b - \bmt)_{-})  }
\]
\[
 = - (\pst(b) \om , (b - \bmt)_{-}) + ( \mup \bk{D(v)}, (b - \bmt)_{-}).
\]
The first term on the left-hand side is equal to
\[
\jd \ddt \ndk{ (b - \bmt)_{-}} - \n{\frac{\oma \bmi}{\n{1 + \oma \kd t}^{\frac{1}{\kd} + 1}}, (b - \bmt)_{-}}.
\]
The second term of the left-hand side vanishes and the third is nonnegative. Thus, we get
\[
\jd \ddt \ndk{ (b - \bmt)_{-}}
 - \n{\frac{\oma \bmi}{\n{1 + \oma \kd t}^{\frac{1}{\kd} + 1}}, (b - \bmt)_{-}}
\leq - (\pst(b) \om , (b - \bmt)_{-}).
\]
Using  (\ref{ge}) we get
\[
\jd \ddt \ndk{ (b - \bmt)_{-}}
 - \n{\frac{\oma \bmi}{\n{1 + \oma \kd t}^{\frac{1}{\kd} + 1}}, (b - \bmt)_{-}}
 \]
 \[
\leq - \frac{\oma}{1 + \oma \kd t} (\pst(b) , (b - \bmt)_{-})
\]
and by definition (\ref{h}) we obtain
\[
\jd \ddt \ndk{ (b - \bmt)_{-}}
\leq - \frac{\oma}{1 + \oma \kd t} (\pst(b) - \bmt, (b - \bmt)_{-}).
\]
From (\ref{gc}) and  (\ref{estipsi}) we have $\pst(b)\leq b$ so, we obtain
\[
\jd \ddt \ndk{ (b - \bmt)_{-}}
\leq - \frac{\oma}{1 + \oma \kd t} (b - \bmt, (b - \bmt)_{-})
\]
\[
= - \frac{\oma}{1 + \oma \kd t} \ndk{ (b-\bmt)_{-} }
\]
and then  $\ddt \ndk{ (b - \bmt)_{-} }\leq 0$. Using (\ref{f}) and (\ref{h}) we get  (\ref{gf}).

Note that from (\ref{defpsi}) and (\ref{gf}) we get
\eqq{\pst (b)=b.
}{newb}
Further, (\ref{defPsi}) and  (\ref{gf})  give $\Pst (b)=b$. Finally, (\ref{h}), (\ref{defPhi}), (\ref{le}) and (\ref{ge}) yield  $\Pht(\om) = \om$. Thus,
\eqq{
\mup = \frac{\Pst (b)}{\Pht (\om)} = \frac{b}{\om}.
}{ha}
Applying (\ref{newa}), (\ref{newb}) and (\ref{ha})  we deduce that system (\ref{gb})-(\ref{ccg}) has the following form
\eqq{(\vt, w) -  (v\otimes v,\nabla w)  +     \n{ \frac{b}{\om}  D(v), D(w)}   =0 \hd \m{ for } \hd w\in \Vkdj,}{gbh}
\eqq{ (\ot, z) - (\om v, \nabla z  ) + \n{ \frac{b}{\om} \nabla \om, \nabla z  } = - \kd (\om^{2}, z ) \hd \m{ for } \hd z\in \Vj,  }{cbgh}
\eqq{(\bt, q) - (b v, \nab q   )   +\n{  \frac{b}{\om}  \nabla b, \nabla q  }  = - (b \om , q) + \left( \frac{b}{\om} \bk{D(v)}, q\right) \hd \m{ for } \hd q\in \Vj   }{ccgno}
for a.a. $t\in (0,\ts)$.

\section{Appendix }

The function $\Pst $ may be defined as follows. We set $f(x)= e^{-1/x}$ for $x>0$ and zero elsewhere. We put $g(x)= x-e^{-1/x}$ for $x<0$ and $g(x)=x$ for $x>0$. Then we set
\[
\tilde{\eta}(x) = \frac{1}{c} \int_{0}^{x} f(y)f(-y+1)dy,
\]
where $c=\int_{0}^{1} f(y)f(-y+1)dy$. Function $\tilde{\eta } $ is smooth function, which vanishes for negative $x$ and is equal to one for $x>1$. Next, we put
\[
\eta(x)= \tilde{\eta}(2(x-\frac{1}{4})), \hd \hd h(x)= (1- \eta(x))f(x)+ \eta(x)g(x).
\]
Finally, we define
\eqq{\Pst(x) = \frac{\bmt}{2} + \frac{\bmt}{2} h\left(  \frac{2}{\bmt} \left( x- \frac{\bmt}{2} \right) \right).    }{defPsik}
{\bf{Acknowledgements}} The authors would like to thank the anonymous referee for valuable remarks, which significantly improve the paper.

\end{document}